\documentclass[fleq]{amsart}
\usepackage{amssymb,amsmath,latexsym,amsbsy,color,enumerate}
\numberwithin{equation}{section}

\newtheorem{theorem}{Theorem}

\newtheorem{lemma}{Lemma}

\begin{document}
 \title{Degree complexity of a family of birational maps: II. Exceptional cases}
\author{Tuyen Trung Truong}
  \address{Department of mathematics, Indiana University Bloomington, IN 47405}
 \email{truongt@indiana.edu}

\thanks{The author would like to thank Professor Bedford for his helpful suggestions.}
    \date{\today}
    \begin{abstract}
We compute the degree complexity of the family of birational maps considered in \cite{bedford-kim-tuyen-abarenkova-maillard} for all exceptional cases.
Some interesting properties of the family are also given.
\end{abstract}
\maketitle
\section{Introduction}
We continue the work of \cite{bedford-kim-tuyen-abarenkova-maillard} where we considered a family $k_F$ of birational maps of the plane determined by a
choice of polynomial $F(w) = a_0 + a_1w +\ldots + a_nw^n$ (the definition of the family $k_F$ will be recalled in Section 2). In \cite{bedford-kim-tuyen-abarenkova-maillard} we determined the degree complexity $\delta
(k_F)$ in the generic case. If $\delta (k_{\widehat{F}})$ is less than the generic value, then we say that $\widehat{F}$ is exceptional. (The set of
exceptional parameters is a nowhere dense algebraic subset.) This corresponds to cases of degree reduction which are especially interesting because they
correspond to the maps that have special symmetries.

As seen in \cite{bedford-kim-tuyen-abarenkova-maillard}, there is a fundamental difference between the cases where n, the degree of F, is even or odd.

The complexity degree $\delta (k_F)$ in the case $n$ is even is given by
\begin{theorem}
Suppose that $n=deg(F)$ is even.

Case 1: If $a_0\not= \frac{2}{1+m}$ for all integers $m\geq 0$ then $\delta (k_F)$ is the largest root of $x^2-(n+1)x-1$.

Case 2: If $a_0=\frac{2}{1+m}$ for some integer $m\geq 0$ then $\delta (k_F)$ is the largest real root of the polynomial $x^{2m+1}(x^2-(n+1)x-1)+x^2+n.$
\label{DegreeEvenCase}\end{theorem} The complexity degree $\delta (k_F)$ in the case $n$ is odd is given by
\begin{theorem}
Suppose that $n=deg(F)\geq 3$ is odd. Let linear functions $L_j:\mathbb C^{n+1}\rightarrow \mathbb C$ (for $0\leq j\leq n$) be defined as in
(\ref{LinearFunction})
\begin{eqnarray*}
L_j(a_0,a_1,\ldots ,a_n)=-(a_{n-j}+a_{n-j+1})-[-a_n\left (\stackrel{n}{j}\right )+a_{n-1}\left (\stackrel{n-1}{j-1}\right )+\ldots
+(-1)^{j+1}a_{n-j}\left (\stackrel{n-j}{0}\right )],
\end{eqnarray*}
where
\begin{eqnarray*}
\left (\stackrel{n}{j}\right )=\frac{n!}{j!(n-j)!}
\end{eqnarray*}
with $n!=n(n-1)\ldots 2.1$ the factorial of $n$.

Let $0\leq h\leq n-2$ be the largest integer in $[0,n-2]$ for which
\begin{eqnarray*}
L_j(a_0,a_1,\ldots ,a_n)=0
\end{eqnarray*}
for all $0\leq j\leq h$.

Case 1: $h<n-2$, and $a_0\not= 2/(1+m)$ for all integers $m\geq 0$. Then $\delta (k_F)$ is the largest real root of the polynomial $x^3-nx^2-(n+1-h)x-1.$

Case 2: $h<n-2$, and $a_0=2/(1+m)$ for some integer $m\geq 0$. Then $\delta (k_F)$ is the largest real root of the polynomial $x^{2m+1}(x^3-nx^2-(n-h+1)x-1)+x^3+x^2+nx+n-h-1.$

Case 3: $h=n-2$, and $a_0\not= 2/(1+m)$ for all integers $m\geq 0$, and $a_0\not= \frac{n+1}{2}+\frac{l}{2(1+l)}$ for all integers $l\geq 0$. Then
$\delta (k_F)$ is the largest real root of the polynomial $x^3-nx^2-2x-1.$

Case 4: $h=n-2$, and $a_0= 2/(1+m)$ for some integer $m\geq 0$, and $a_0\not= \frac{n+1}{2}+\frac{l}{2(1+l)}$ for all integers $l\geq 0$. Then $\delta
(k_F)$ is the largest real root of the polynomial $x^{2m}(x^3-nx^2-2x-1)+x^2+x+n.$

Case 5:  $h=n-2$, and $a_0\not= 2/(1+m)$ for all integers $m\geq 0$, and $a_0= \frac{n+1}{2}+\frac{l}{2(1+l)}$ for some integer $l\geq 0$. Then $\delta
(k_F)$ is the largest real root of the polynomial $x^{2l+2}(x^3-nx^2-2x-1)+nx^2+x+1.$

Case 6: $h=n-2$, and $a_0= 2/(1+m)$ for some integer $m\geq 0$, and $a_0= \frac{n+1}{2}+\frac{l}{2(1+l)}$ for some integer $l\geq 0$. Then $n=3$,
$a_0=2$, and the map $k_F$ is exactly the family considered in Section 5 in \cite{bedford-kim-tuyen-abarenkova-maillard}. Hence in this case $\delta
(k_F)=1$. \label{DegreeOddCase}\end{theorem} There are two interesting phenomena which occur to the maps $k_F$:

1. The first phenomenon, which occurs when $n\geq 3$ is odd, is what we call "double point-blowups". This means that in Theorem \ref{DegreeOddCase}, if
$h<n-2$ then $h$ is an even number, while if $h=n-2$ then $L_{n-1}(a_0,\ldots ,a_n)=0$. We will give an example exploring the case $n=3$ in Section \ref{sectionExamplen=3} to
illustrate this phenomenon. This is a consequence of the results about a system of linear equations that we will explore in Section \ref{sectionSystemLinearEquations}.

2. The other phenomenon is that there is no automorphism if $n=deg(F)$ is different from $1$ or $3$. This is also a sequence of the results about a system
of linear equations that we mentioned above. The exact formulation of this phenomenon is
\begin{theorem}
Let $n=deg(F)$. If $n\not= 1$ and $n\not= 3$ then there is no space $Z$ which satisfies the following two conditions:

1) $Z$ is constructed from $\mathbb P^2$ by a finite number of point blowing-ups.

2) The induced map $k_Z:Z\rightarrow Z$ is an automorphism.

If $n=3$ then a space $Z$ with properties 1) and 2) exists iff $F(z)=a_3z^3+a_3z^2+a_1z+2$.
\label{NoAutomorphism}\end{theorem}
\begin{proof}
We consider three cases:

Case 1: $n=deg(F)$ is even. Then from the proof of Theorem \ref{DegreeEvenCase}, it follows that we can not resolve the point $\frac{1}{a_n}\in P_{n-1}$, which
is the image of some exceptional curves, to obtain an automorphism.

Case 2: $n=deg(F)\geq 5$ is odd. Then it follows from the proof of Theorem \ref{OddCase}, we can have an automorphism iff simultaneously
$a_0=\frac{2}{m+1}$ for some $m=0,1,2,\ldots $, and $a_0=\frac{n+1}{2}+\frac{l}{2(l+1)}$ for some $l=0,1,2,\ldots $. If $a_0=\frac{2}{m+1}$ then
obviously $a_0\leq 2$, while if $a_0=\frac{n+1}{2}+\frac{l}{2(l+1)}$ and $n\geq 5$ then $a_0\geq 3$. Hence in this case we do not have an automorphism.

Case 3: $n=deg(F)=3$. Then use Lemma \ref{Automorphism} we have that a space $Z$ with properties 1) and 2) in the statement of Theorem
\ref{NoAutomorphism} exists iff $F(z)=a_3z^3+a_3z^2+a_1z+2$.
\end{proof}

Theorem \ref{NoAutomorphism} shows that the family $k_F$ corresponding with the maps $F(z)=az^3+az^2+b+2$ (with $a\not= 0$) as described in Section 5 in
\cite{bedford-kim-tuyen-abarenkova-maillard} is the only family of automorphism in the whole family $k_F$, besides the case $n=deg(F)=1$ which was known
previously (see \cite{diller-favre} and \cite{bedford-diller}).
\section{Properties of $k_F$}
We review in this Section some results in \cite{bedford-kim-tuyen-abarenkova-maillard}.

Let $\mathbb P^2$ be the complex projective space of dimension $2$, with coordinate $[x_0:x_1:x_2]$. Given a polynomial of degree $n$
\begin{eqnarray*}
F(z)=a_nz^n+a_{n-1}z^{n-1}+\ldots +a_1z+a_0
\end{eqnarray*}
where $a_n\not= 0$, we define a birational map $k=k_F:\mathbb P^2\rightarrow \mathbb P^2$ (see \cite{bedford-kim-tuyen-abarenkova-maillard}) by the
formula $k_F=j_F\circ i$ where $j_F$ and $i$ are involutions defined in the open Zariski dense set $\mathbb C^2$ of $\mathbb P^2$ by
\begin{eqnarray*}
j_F(x,y)=(-x+F(y),y),~i(x,y)=(1-x-\frac{x-1}{y},-y-1-\frac{y}{x-1}).
\end{eqnarray*}
The map $k_F=[k_0:k_1:k_2]$ is given in homogeneous coordinates as
\begin{eqnarray*}
k_0&=&(x_0x_1-x_0^2)^nx_2,\\
k_1&=&x_0^{n-1}(x_0-x_1)^{n+1}(x_0+x_2)+x_2\sum _{j=0}^na_j(x_0x_1-x_0^2)^{n-j}(x_2^2-x_0x_1-x_1x_2)^j,\\
k_2&=&x_2(x_0x_1-x_0^2)^{n-1}(x_2^2-x_0x_1-x_1x_2).
\end{eqnarray*}
It is worth to write out also the non-homogeneous form of $k$ which is convenient in computation
\begin{equation}
k[x_0:x_1:x_2]=[1:\frac{(x_1-x_0)(x_2+x_0)}{x_0x_2}+F(-1-\frac{x_1x_2}{x_0(x_1-x_0)}):-1-\frac{x_1x_2}{x_0(x_1-x_0)}]\label{TheMap}
\end{equation}
The inverse map is
\begin{equation}
k^{-1}[x_0:x_1:x_2]=[1:(1+\frac{x_1}{x_0}-F(\frac{x_2}{x_0}))(1+\frac{x_0}{x_2}):-1-\frac{x_2}{x_0}+\frac{x_2}{x_0+x_1-x_0F(x_2/x_0)}].\label{TheInverseMap}
\end{equation}
We recall the following notations from \cite{bedford-kim-tuyen-abarenkova-maillard}:
\begin{eqnarray*}
C_1&=&\{x_0=0\}, ~C_2=\{x_0=x_1\},~C_3=\{x_2=0\},C_4=\{-x_0^2+x_0x_1+x_1x_2=0\},\\
C_1'&=&C_1,~C_2'=\{1+\frac{x_1}{x_0}-F(\frac{x_2}{x_0})=0\},~C_3'=C_3,\\
C_4'&=&\{\frac{x_2}{x_0}-(1+\frac{x_2}{x_0})(1+\frac{x_1}{x_0}-F(\frac{x_2}{x_0}))=0\}.
\end{eqnarray*}
The exceptional hypersurfaces of $k_F$ are mapped as
\begin{eqnarray*}
k_F:~C_4\mapsto [1:-1+a_0:0]\in C_3,~C_1\cup C_2\cup C_3\mapsto e_1.
\end{eqnarray*}
The points of indeterminate of $k_F$ are $e_1=[0:1:0],e_2=[0:0:1],$ and $e_{01}=[1:1:0]$. The exceptional curves for $k_F^{-1}$ are mapped as
\begin{eqnarray*}
k_F^{-1}:~C_1'\cup C_3'\mapsto e_1,~C_2'\mapsto e_2,~C_{4}'\mapsto e_{01}.
\end{eqnarray*}
Notation: Let $Z$ be a complex manifold and let $k_Z:Z\rightarrow Z$ be a birational map. Recall that $k_Z:Z\rightarrow Z$ is (1,1)-regular or
algebraically stable (or A.S. for brevity) if it satisfies
\begin{equation}
(k_Z^p)^*=(k_Z^*)^p\label{ASCondition}
\end{equation}
for all $p\in\mathbb N$, where $k_Z^*:Pic(Z)\rightarrow Pic(Z)$ is the induced pull-back of $k_Z$ on the Picard group of $Z$, and similarly for
$(k_Z^p)^*$. In \cite{diller-favre}, it was proved that any birational map of a compact Kahler surface can be (1,1)-regularized after a finite number of
point-blowups.
\section{Proof of Theorem \ref{DegreeEvenCase}}\label{sectionDegreeEvenCase}
First we recall construction of the space $X$ constructed in Section 3 of \cite{bedford-kim-tuyen-abarenkova-maillard}: Define a complex manifold $\pi
_X:X\rightarrow \mathbb P^2$ (see Figure 3.1 in \cite{bedford-kim-tuyen-abarenkova-maillard}) by blowing ups points $e_1,p_1,\ldots ,p_{n-1}$ in the
following order:

i) blowup $e_1=[0:1:0]$ and let $E_1$ denote the exceptional fiber over $e_1$,

ii) blowup $q=E_1\cap C_4$ and let $Q$ denote the exceptional fiber over $q$,

iii) blowup $p_1=E_1\cap C_1 $ and let $P_1$ denote the exceptional fiber over $e_1$,

iv) blowup $p_j=P_j\cap E_1$ with exceptional fiber $P_j$ for $2\leq j\leq n-1$.

The exceptional curves of the induced map $k_X$ are $C_1,C_2,C_4,P_1,\ldots ,P_{n-2}$. All the curves $C_1,C_2,P_1\ldots ,P_{n-2}$ are mapped to the same
point $1/a_n\in P_{n-1}$, while $C_4$ is mapped to the point $[1:-1+a_0:0]\in C_3$. By Lemmas 3.2 and 3.3 in
\cite{bedford-kim-tuyen-abarenkova-maillard}, the only way that an exceptional curve can be mapped to a point of indeterminacy is if $a_0=2/(m+1)$ for
some $m\in \mathbb N$, and in this case we have $k_X^{2m+1}C_4=[1:1:0]$.

If $a_0=2/(m+1)$ we construct the new manifold $Z$ by blowing up the manifold $X$ at the points
\begin{eqnarray*}
r_0&=&[1:-1+a_0:0]\in C_3,\\
q_1&=&k_X(r_0)\in Q,~r_1=k_X(q_1)\in C_3,\\
\hdots \\
q_m&=&k_X(r_{m-1})\in Q,~r_m=k_X(q_m)=[1:1:0]\in C_3.
\end{eqnarray*}
Call $R_0,Q_1,R_1,\ldots ,Q_m,R_m$ the exceptional fibers of this blowup.
\begin{lemma}
If $a_0=2/(m+1)$ and $Z$ is constructed as above then the curves $C_4,R_0,Q_1,R_1,\ldots ,Q_m,R_m$ are not exceptional for
$k_Y$.\label{ASEvenCase}\end{lemma}
\begin{proof}
It suffices to check that $C_4$ is not exceptional. We choose a local projection for $R_0$ as
\begin{eqnarray*}
Z\ni (s,u)\mapsto [1:-1+a_0+su:s].
\end{eqnarray*}
In this coordinate chart $R_0=\{s=0\}$. If we rewrite $k[x_0:x_1:x_2]$  as
\begin{eqnarray*}
k[x_0:x_1:x_2]=[1:-1-\frac{x_0^2-x_0x_1-x_1x_2}{x_0x_1}+F(\frac{x_0^2-x_0x_1-x_1x_2}{x_0(x_1-x_0)}):\frac{x_0^2-x_0x_1-x_1x_2}{x_0(x_1-x_0)}]
\end{eqnarray*}
then it can be seen that
\begin{eqnarray*}
k_Z:C_4\ni [x_0:x_1:x_2]\mapsto -1+\frac{x_0}{x_1}\in R_0.
\end{eqnarray*}
Hence $C_3$ is not exceptional.
\end{proof}
It follows (see \cite{diller-favre}) that $k_Z$ is A.S. Thus we obtain $\delta (k_F)$ as the spectral radius of $k_Z^*$. Now we compute $k_Z^*$.

For brevity, we denote by $E_1$ the strict transform $\widetilde{E_1}$ in $Z$ of the exceptional fiber $E_1$, and the same notation $E_1$ is also
used for the class in $Pic(Z)$ of $\widetilde{E_1}$. The same convenience is applied to $C_1,C_2,C_3,P_1,\ldots ,P_{n-1},Q,Q_1,\ldots ,Q_{m},R_0,\ldots
,R_{m}$. Let $H_Z$ denote the class in $Pic (Z)$ of the strict transform of a generic line $H$ in $\mathbb P^2$. Then $H_Z,E_1,P_1,\ldots
,P_{n-1},Q,Q_1,\ldots ,Q_{m},R_0,\ldots ,R_{m}$ form a basis for the space $Pic(Z)$. $C_1,C_2,C_3,C_4$ can be represented in this basis as
\begin{eqnarray*}
C_1&=&H_Z-E_1-Q-\sum _{j=1}^{n-1}(j+1)P_j-\sum _{j=1}^mQ_j,\\
C_2&=&H_Z-R_m,\\
C_3&=&H_Z-E_1-Q-\sum _{j=1}^{n-1}jP_j-\sum _{j=1}^mQ_j-\sum _{j=0}^mR_j,\\
C_4&=&2H_Z-E_1-2Q-\sum _{j=1}^{n-1}jP_j-2\sum _{j=1}^mQ_j-R_m.
\end{eqnarray*}
The induced map $k_Z$ acts as follows
\begin{eqnarray*}
&&k_Z:E_1\mapsto E_1,~P_{n-1}\mapsto P_{n-1},~C_1,C_2,P_1,\ldots ,P_{n-2}\mapsto \frac{1}{a_n}\in P_{n-1},~Q\mapsto C_3\mapsto Q,\\
&&k_Z:C_4\mapsto R_0\mapsto Q_1\mapsto R_1\mapsto \ldots \mapsto Q_m\mapsto R_m\mapsto C_4',\\
&&k_{Z}^{-1}:C_1,P_1,\ldots ,P_{n-1}\mapsto -\frac{1}{a_n}\in P_{n-1}.
\end{eqnarray*}
From this, the induced map $k_Z^*:Pic(Z)\rightarrow Pic(Z)$ is as follows
\begin{eqnarray*}
k_Z^*(H_Z)&=&(2n+1)H_Z-nE_1-(n+1)Q-(n+1)\sum _{j=1}^{n-1}jP_j-(n+1)\sum _{j=1}^mQ_j-(n+1)R_m,\\
k_Z^*(E_1)&=&E_1,\\
k_Z^*(Q)&=&C_3=H_Z-E_1-Q-\sum _{j=1}^{n-1}jP_j-\sum _{j=1}^mQ_j-\sum _{j=0}^mR_j,\\
k_Z^*(P_j)&=&0,~1\leq j\leq n-2,\\
k_Z^*(P_{n-1})&=&C_1+C_2+\sum _{j=1}^{n-1}P_j=2H_Z-E_1-Q-\sum _{j=1}^{n-1}jP_j-\sum _{j=1}^mQ_j-R_m,\\
k_Z^*(R_0)&=&C_4=2H_Z-E_1-2Q-\sum _{j=1}^{n-1}jP_j-2\sum _{j=1}^mQ_j-R_m,\\
k_Z^*(R_j)&=&Q_j,~1\leq j\leq m,~k_Z^*(Q_j)=R_{j-1},~1\leq j\leq m.
\end{eqnarray*}
From the above we find that the characteristic polynomial of $k_Z^*$ is $$P(x)=-x[x^{2m+1}(x^2-(n+1)x-1)+x^2+n].$$ From this, Theorem \ref{DegreeEvenCase} follows.
\section{Example: case $n=3$}\label{sectionExamplen=3}
In this section we explore the map $k$ when $n=deg(F)=3$. In this case $F(z)=a_3z^3+a_2z^2+a_1z+a_0$ where $a_3\not= 0$.

Let $Y_1=Y$ be the manifold and $E,Q,P_1,P_2,P_3$ the exceptional fibers constructed in Section 4 in \cite{bedford-kim-tuyen-abarenkova-maillard}. The
action of the induced map $k_{Y_1}:Y_1\rightarrow Y_1$ is
\begin{eqnarray*}
C_1,C_2\stackrel{k_{Y_1}}{\mapsto}-\frac{a_2}{a_3^2}\in P_3,~C_1\stackrel{k_{Y_1}^{-1}}{\mapsto}\frac{-2a_3+a_2}{a_3^2}\in P_3,
\end{eqnarray*}
and
\begin{eqnarray*}
k_{Y_1}:P_1\ni u\mapsto -\frac{(1+a_2u)}{a_3^2u}\in P_3,~P_3\ni u\mapsto \frac{1}{-2a_3+a_2-a_3^2u}\in P_1.
\end{eqnarray*}
(In these formula, we use the same local coordinates as that of \cite{bedford-kim-tuyen-abarenkova-maillard}.)

1) Case 1: $a_2\not=a_3$. Then the orbit of exceptional curves $C_1,C_2$ will never land on an indeterminacy point. Hence depending on whether
$a_0=\frac{2}{m+1}$ for some $m=0,1,2,\ldots$ or not, we will decide to perform the blowups as in the proof of Theorem \ref{DegreeEvenCase} or not. For the resulting
manifold $Z$, the induced map $k_Z$ is A.S.

2) Case 2: $a_2=a_3$. Then the map $k_{Y_1}$ is not A.S because
\begin{eqnarray*}
\frac{-2a_3+a_2}{a_3^2}=-\frac{a_2}{a_3^2}=-\frac{1}{a_3},
\end{eqnarray*}
that is the point $-\frac{1}{a_3}\in P_3$ is both an indeterminate point and the image of exceptional curves $C_1,C_2$. We blowup the space $Y_1$ at the
point $-\frac{1}{a_3}\in P_3$. Call $Y_2$ the resulting manifold and $P_4$ the exceptional fiber of this blowup. We choose a coordinate projection for
$P_4$ as
\begin{eqnarray*}
Y_2\ni (s,u)\mapsto [s^3(\frac{1}{a_3}-\frac{s}{a_3}+s^2u):1:s^2(\frac{1}{a_3}-\frac{s}{a_3}+s^2u)].
\end{eqnarray*}
Recall that this means that in this local coordinate $P_4$ is given by the equation $s=0$.

Then the action of the induced map $k_{Y_2}:Y_2\rightarrow Y_2$ is
\begin{eqnarray*}
k_{Y_2}:~P_4\ni u&\mapsto& [0:1:\frac{1}{-a_3+a_1+a_3^2u}]\in C_1,\\
k_{Y_2}:~C_1\ni [0:1:u]&\mapsto&\frac{1+a_3u-a_1u}{a_3u}\in P_4,
\end{eqnarray*}
and
\begin{eqnarray*}
k_{Y_2}:~C_2\mapsto [\frac{a_3-a_1}{a_3^2}]_{P_4}\mapsto [0:0:1]=e_2,
\end{eqnarray*}
Since $e_2$ is an indeterminate point, it follows that $k_{Y_2}$ is not A.S., and we need to blowup more times. This is what we called "double
point-blowups" in Section 1.

We blowup $Y_2$ at points $\frac{a_3-a_1}{a_3^2}\in P_4$ and $e_2$. Call $Y_3$ the resulting manifold and $P_5,E_2$ the exceptional fibers of this
blowup.

We choose a coordinate projection for $P_5$ as
\begin{eqnarray*}
Y_3\ni (s,u)\mapsto [s^3(\frac{1}{a_3}-\frac{s}{a_3}+\frac{a_3-a_1}{a_3^2}s^2+s^3u):1:s^2(\frac{1}{a_3}-\frac{s}{a_3}+\frac{a_3-a_1}{a_3^2}s^2+s^3u)],
\end{eqnarray*}
and a coordinate projection for $E_2$ as
\begin{eqnarray*}
Y_3\ni (s,u)\mapsto [s:su:1].
\end{eqnarray*}
The action of the induced map $k_{Y_3}$ is
\begin{eqnarray*}
k_{Y_3}:~P_5\ni u&\mapsto&-a_3^2u-a_3+2a_1+a_0-4\in E_2,\\
k_{Y_3}:E_2\ni u&\mapsto&\frac{u+a_3-2a_1+a_0-1}{a_3^2}\in P_5,\\
k_{Y_3}^2:E_2\ni u&\mapsto& u+2a_0-5\in E_2,
\end{eqnarray*}
and
\begin{eqnarray*}
k_{Y_3}:~C_2\mapsto -\frac{a_3-2a_1+a_0}{a_3^2}\in P_5\mapsto 2a_0-4\in E_2.
\end{eqnarray*}
This map has only one more indeterminate point at $0\in E_2$.

2.1) Subcase 2.1: $a_0\not= 2+\frac{l}{2(l+1)}$ for any $l=0,1,2,\ldots $. Then the orbit of the exceptional curve $C_2$ will never land on an
indeterminacy point. Hence depending on whether $a_0=\frac{2}{m+1}$ for some $m=0,1,2,\ldots$ or not, we will decide to perform the blowups as in Theorem
\ref{DegreeEvenCase} or not. For the resulting manifold $Z$, the induced map $k_Z$ is A.S.

2.2) Subcase 2.2: $a_0=2+\frac{l}{2(l+1)}$ for some $l=0,1,2,\ldots $. In this case we do a series of blowups at the point $0\in E_2$ and a finite number
of its previous images in the same way as we did in Theorem \ref{DegreeEvenCase}. If also $a_0=\frac{2}{m+1}$ for some $m=0,1,2,\ldots$ we also perform the
series of blowups in Theorem \ref{DegreeEvenCase}. For the resulting space $Z$, the induced map $k_Z:Z\rightarrow Z$ is A.S.

Note that if both $a_0=2+\frac{l}{2(l+1)}$ for some $l=0,1,2,\ldots $ and $a_0=\frac{2}{m+1}$ for some $m=0,1,2,\ldots$ then $a_0=2$. In particular, from
the above analysis, we get that
\begin{lemma}
If $n=3$, then a space $Z$ satisfying 1) and 2) of Theorem \ref{NoAutomorphism} exists iff $F(z)=a_3z^3+a_3z^2+a_1z+2$ where $a_3\not= 0$.
\label{Automorphism}\end{lemma}
\section{A system of linear equations}\label{sectionSystemLinearEquations}
In this section we explore a system of linear equations which is related to the map $k_F$.

Fixed $n\in \mathbb N$, where $n$ is not necessarily odd. We define linear functions $L_j:\mathbb C^{n+1}\rightarrow \mathbb C$ as follows
\begin{equation}
L_j(a_0,a_1,\ldots ,a_n)=-(a_{n-j}+a_{n-j+1})-[-a_n\left (\stackrel{n}{j}\right )+a_{n-1}\left (\stackrel{n-1}{j-1}\right )+\ldots
+(-1)^{j+1}a_{n-j}\left (\stackrel{n-j}{0}\right )],\label{LinearFunction}
\end{equation}
for $0\leq j\leq n$, and where
\begin{eqnarray*}
\left (\stackrel{n}{j}\right )=\frac{n!}{j!(n-j)!}
\end{eqnarray*}
with $n!=n(n-1)\ldots 2.1$ the factorial of $n$. Functions $L_j$ for some first values of $j$ are:
\begin{eqnarray*}
L_0&=&-a_n-[-a_n]=0,\\
L_1&=&-(a_n+a_{n-1})-[-na_n+a_{n-1}]=(n-1)a_n-2a_{n-1},\\
L_2&=&-(a_{n-1}+a_{n-2})-[-a_n\left (\stackrel{n}{2}\right )+a_{n-1}\left (\stackrel{n-1}{1}\right )-a_{n-2}\left (\stackrel{n-2}{0}\right
)]=\frac{n}{2}L_1.
\end{eqnarray*}
We will explore the properties of systems of linear equations of the form
\begin{equation}
L_j(a_0,a_1,\ldots ,a_n)=0\label{LinearSystem0}
\end{equation}
for all $j=0,1,2,\ldots ,m$, where $0\leq m< n$ is a constant integer. It will be convenient to write equations (\ref{LinearSystem0}) as
\begin{equation}
-(a_{n-j}+a_{n-j+1})=-a_n\left (\stackrel{n}{j}\right )+a_{n-1}\left (\stackrel{n-1}{j-1}\right )+\ldots +(-1)^{j+1}a_{n-j}\left (\stackrel{n-j}{0}\right
)\label{LinearSystem1}
\end{equation}
Changing the order of indexes, the equations (\ref{LinearSystem1}) can be written in a more convenient form
\begin{equation}
-(b_j+b_{j-1})=-b_0\left (\stackrel{n}{j}\right )+b_{1}\left (\stackrel{n-1}{j-1}\right )+\ldots +(-1)^{j+1}b_{j}\left (\stackrel{n-j}{0}\right
).\label{LinearSystem}
\end{equation}
The following results will be used to prove the phenomenon "double point-blowups" that we mentioned in Section 1.
\begin{lemma}
If $0\leq m< n$, and $m$ is odd, and if $b_0,b_1,\ldots ,b_n$ satisfy the equations (\ref{LinearSystem}) for all $j=1,3,5,\ldots ,m$ then $b_0,b_1,\ldots
,b_n$ also satisfy (\ref{LinearSystem}) for all $j=0,2,4,\ldots ,m+1$. \label{problem1}\end{lemma}
\begin{proof}
Fixed $0\leq m< n$, where $m$ is odd. Let $b_0,b_1,\ldots ,b_n$ satisfy the equations (\ref{LinearSystem}) for all $j=1,3,5,\ldots ,m$. To prove Lemma
\ref{problem1} it suffices to prove the following claim:

Claim 1: $b_0,b_1,\ldots ,b_n$ also satisfy (\ref{LinearSystem}) for $j=m+1$.

The proof is divided in several steps.

i) Reduction 1: In equations (\ref{LinearSystem}) with $j=1,3,\ldots ,m$, pushing all $b_i$ with $i$ odd to the left hand-sided and pushing all $b_i$
with $i$ even to the right hand-sided we can rewrite them as
\begin{eqnarray*}
&&2b_1=b_0\left (\stackrel{n-1}{1}\right ),\\
&&b_1\left (\stackrel{n-1}{2}\right )+2b_3=b_0\left (\stackrel{n}{3}\right )+b_2\left (\stackrel{n-3}{1}\right ),\\
&&b_1\left (\stackrel{n-1}{4}\right )+b_3\left (\stackrel{n-3}{2}\right )+2b_5=b_0\left (\stackrel{n}{5}\right )+b_2\left (\stackrel{n-2}{3}\right
)+b_4\left (\stackrel{n-5}{1}\right ),\\
&&\vdots \\
&&b_1\left (\stackrel{n-1}{m-1}\right )+b_3\left (\stackrel{n-3}{m-3}\right )+\ldots +b_{m-2}\left (\stackrel{n-m+2}{2}\right )+2b_{m}\\
&&=b_0\left (\stackrel{n}{m}\right )+b_{2}\left (\stackrel{n-2}{m-2}\right )+\ldots +b_{m-3}\left (\stackrel{n-m+3}{3}\right )+b_{m-1}\left
(\stackrel{n-m}{1}\right ).
\end{eqnarray*}

The equation (\ref{LinearSystem}) for $j=m+1$ which we want to prove in Claim 1 can be written as
\begin{eqnarray*}
&&b_1\left (\stackrel{n-1}{m}\right )+b_3\left (\stackrel{n-3}{m-2}\right )+\ldots +b_{m-2}\left (\stackrel{n-m+2}{3}\right )+b_{m}\left
(\stackrel{n-m+1}{1}\right )\\
&=&b_0\left (\stackrel{n}{m+1}\right )+b_2\left (\stackrel{n-2}{m-1}\right )+\ldots +b_{m-1}\left (\stackrel{n-m+1}{2}\right ).
\end{eqnarray*}

ii) Reduction 2: For any value of $b_0,b_2,b_4,\ldots ,b_{m-1}$ there exists a unique solution $b_1,b_3,\ldots ,b_m$ to the system (\ref{LinearSystem})
for $j=1,3,\ldots ,m$. For a proof of this claim we can use the rewritten system in Reduction 1.

iii) Reduction 3: Claim 1 is true in general case if we can prove it is true for the special case $b_0=1,b_2=b_4,\ldots =0$. For a proof use the special
structure of the rewritten system in Reduction 1.

From now on in this proof we will assume that $b_0=1,b_2=b_4=\ldots =0$. We rewrite Reduction 1 as

iv) Reduction 4: In equations (\ref{LinearSystem}) with $j=1,3,\ldots ,m$, pushing all $b_i$ with $i$ odd to the left hand-sided and pushing all $b_i$
with $i$ even to the right hand-sided we can rewrite them as
\begin{eqnarray*}
2b_1&=&\left (\stackrel{n-1}{1}\right ),\\
b_1\left (\stackrel{n-1}{2}\right )+2b_3&=&\left (\stackrel{n}{3}\right ),\\
b_1\left (\stackrel{n-1}{4}\right )+b_3\left (\stackrel{n-3}{2}\right )+2b_5&=&\left (\stackrel{n}{5}\right ),\\
&&\vdots \\
b_1\left (\stackrel{n-1}{m-1}\right )+b_3\left (\stackrel{n-3}{m-3}\right )+\ldots +b_{m-2}\left (\stackrel{n-m+2}{2}\right )+2b_{m}&=&\left
(\stackrel{n}{m}\right ).
\end{eqnarray*}

The equation (\ref{LinearSystem}) for $j=m+1$ which we want to prove in Claim 1 can be written as
\begin{eqnarray*}
b_1\left (\stackrel{n-1}{m}\right )+b_3\left (\stackrel{n-3}{m-2}\right )+\ldots +b_{m-2}\left (\stackrel{n-m+2}{3}\right )+b_{m}\left
(\stackrel{n-m+1}{1}\right )=\left (\stackrel{n}{m+1}\right ).
\end{eqnarray*}
v) Reduction 5: Define
\begin{eqnarray*}
\beta _1&=&\frac{b_1}{n},\\
\beta _3&=&\frac{b_3}{n(n-1)(n-2)},\\
\beta _5&=&\frac{b_5}{n(n-1)(n-2)(n-3)(n-4)},\\
&&\ldots
\end{eqnarray*}
then $\beta _!,\beta _3,\beta _5,\ldots $ satisfy the following system of equations
\begin{eqnarray*}
2\beta _1&=&1-\frac{1}{n},\\
\frac{\beta _1}{2!}+2\beta _3&=&\frac{1}{3!},\\
\frac{\beta _1}{4!}+\frac{\beta _3}{2!}+2\beta _5&=&\frac{1}{5!},\\
&&\ldots ,\\
\frac{\beta _1}{(m-1)!}+\frac{\beta _3}{(m-3)!}+\ldots \frac{\beta _{m-2}}{2!}+2\beta _{m}&=&\frac{1}{m!}.
\end{eqnarray*}
What we want to prove in Claim 1 can be written as
\begin{eqnarray*}
\frac{\beta _1}{m!}+\frac{\beta _{3}}{(m-2)!}+\ldots +\frac{\beta _{m-2}}{3!}+\beta _{m}(1+\frac{1}{n-m})=\frac{1}{(m+1)!}
\end{eqnarray*}
vi) Reduction 6: A universal system of linear equations

Let $\theta _1, \theta _3,\theta _5,\ldots $ be the unique sequence satisfying the following system of infinitely many linear equations
\begin{eqnarray*}
2\theta _1&=&1,\\
\frac{\theta _1}{2!}+2\theta _3&=&0,\\
\frac{\theta _1}{4!}+\frac{\theta _3}{2!}+2\theta _5&=&0,\\
&&\ldots ,\\
\end{eqnarray*}
Then, for any sequence $c_1,c_3,c_5,\ldots $, the unique solution to
\begin{eqnarray*}
2z _1&=&c_1,\\
\frac{z _1}{2!}+2z_3&=&c_3,\\
\frac{z_1}{4!}+\frac{z _3}{2!}+2z _5&=&c_5,\\
&&\ldots ,\\
\end{eqnarray*}
is
\begin{eqnarray*}
z_1&=&c_1\theta _1,\\
z_3&=&c_3\theta _1+c_1\theta _3,\\
z_5&=&c_5\theta _1+c_3\theta _3+c_5\theta _1,\\
\ldots
\end{eqnarray*}
vii) Reduction 7: Let $\alpha _1,\alpha _3,\ldots $ be the unique sequence satisfying the following system
\begin{eqnarray*}
2\alpha _1&=&\frac{1}{1!},\\
\frac{\alpha _1}{2!}+2\alpha_3&=&\frac{1}{3!},\\
\frac{\alpha _1}{4!}+\frac{\alpha _3}{2!}+2\alpha _5&=&\frac{1}{5!},\\
&&\ldots
\end{eqnarray*}
Then it is easy to see that for $\beta _j$ in Reduction 4:
\begin{eqnarray*}
\beta _j&=&\alpha _j-\frac{1}{n}\theta _j,
\end{eqnarray*}
for all $j=1,3,\ldots ,m$, and what we wanted to prove in Claim 1 becomes
\begin{eqnarray*}
-\frac{1}{n}(\frac{\theta _1}{m!}+\frac{\theta _3}{(m-2)!}+\ldots +\frac{\theta _{m-2}}{3!}+\frac{\theta _m}{1!}-\frac{\theta
_m}{m})+\frac{1}{n-m}(\alpha _m-\frac{\theta _m}{m})=0.
\end{eqnarray*}
Hence Claim 1 is proved if we can prove the following Claim

Claim 2: For any $m\in \mathbb N$, $m$ odd then the following conclusions are true
\begin{equation}
\frac{\theta _1}{m!}+\frac{\theta _3}{(m-2)!}+\ldots +\frac{\theta _{m-2}}{3!}+\frac{\theta _m}{1!}-\frac{\theta _m}{m}=0,\label{Claim2.1}
\end{equation}
and
\begin{equation}
\alpha _m-\frac{\theta _m}{m}=0.\label{Claim2.2}
\end{equation}
viii) Proof of Claim 2:

Define a formal series
\begin{eqnarray*}
\theta (t)=\theta _1-t^2\theta _3+t^4\theta _5-t^6\theta _7+\ldots
\end{eqnarray*}
From the Reduction 6:
\begin{eqnarray*}
1=\theta (t).(2-\frac{t^2}{2!}+\frac{t^4}{4!}-\frac{t^6}{6!}\ldots )=\theta (t).(1+\cos t).
\end{eqnarray*}
Hence
\begin{eqnarray*}
\theta (t)=\frac{1}{1+\cos t}.
\end{eqnarray*}
Similarly, if we define
\begin{eqnarray*}
\alpha (t)=t\alpha _1-t^3\alpha _3+t^5\alpha _5\ldots
\end{eqnarray*}
then from Reduction 7
\begin{eqnarray*}
\alpha (t)=\frac{\sin t}{1+\cos t}.
\end{eqnarray*}
It follows that
\begin{eqnarray*}
\frac{d\alpha}{dt} =\theta (t),
\end{eqnarray*}
which proves (\ref{Claim2.2}).

From Reductions 6 and 7 we have
\begin{eqnarray*}
\alpha _m=\frac{\theta _1}{m!}+\frac{\theta _{3}}{(m-2)!}+\ldots +\frac{\theta _{m-2}}{3!}+\frac{\theta _{m}}{1!}.
\end{eqnarray*}
This equality and (\ref{Claim2.2}) imply (\ref{Claim2.1}). Hence we completed the proof of Lemma \ref{problem1}.
\end{proof}
\begin{lemma}
Let $n\geq 3$ be an odd integer. Let $a_0,\ldots ,a_n$ be a solution of the system of linear equations
\begin{eqnarray*}
L_j(a_0,a_1,\ldots ,a_n)=0
\end{eqnarray*}
for all $j=0,1,2,\ldots ,n-1$. Then
\begin{eqnarray*}
\sum _{j=2}^n(-1)^ja_j=0.
\end{eqnarray*}
\label{problem2}\end{lemma}
\begin{proof}
To prove the equality we need only to take the difference between the sum of odd-th equations and the sum of even-th equations.
\end{proof}
\section{(1,1)-regularization for exceptional cases: $n=$odd}\label{sectionRegularizeOddCase}
In this Section we show how to (1,1)-regularize the maps $k_F$ in exceptional cases when $n\geq 3$ is odd . Recall that $F(z)=a_nz^n+a_{n-1}z^{n-1}+\ldots +a_1z+a_0$ is a
polynomial ($a_n\not= 0$).

Let $Y$ be the manifold constructed in Section 4 in \cite{bedford-kim-tuyen-abarenkova-maillard}. If $a_0=\frac{2}{m+1}$ for some $m=0,1,2,\ldots $, let
$\widehat{Y}$ denote the manifold constructed by blowing up the manifold $Y$ at the points $r_0,q_0,r_1,q_1,\ldots ,r_m,q_m$ as in the proof of Theorem
\ref{DegreeEvenCase}. Otherwise, i.e. if $a_0\not=\frac{2}{m+1}$ for any $m=0,1,2,\ldots $, let $\widehat{Y}$ denote the space $Y$ itself. Let us denote
$Y_1=\widehat{Y}$. Let us also denote some more notations
\begin{eqnarray*}
ep_0&=&\frac{1}{a_n}\in P_{n-1},~ip_0=\frac{1}{a_n}\in P_{n-1},\\
ep_1&=&-\frac{a_{n-1}}{a_n^2}\in P_{n},~ip_1=\frac{-(n-1)a_n+a_{n-1}}{a_n^2}\in P_{n}.
\end{eqnarray*}
The above equations mean that $ep_0$ is a point of $P_{n-1}$ with local coordinate $\frac{1}{a_n}$, and so on. Here "ep" means "exceptional point" that
is points which is the image of some exceptional curves, and "ip" means "indeterminate point" that is points which blowups to some curves.

For convenience we recall the action of the induced map $k_{Y_1}:Y_1\rightarrow Y_1$ (see \cite{bedford-kim-tuyen-abarenkova-maillard}):
\begin{eqnarray*}
k_{Y_1}:~C_1,C_2,P_1,P_2,\ldots ,P_{n-3}&\mapsto&ep_1\in P_n,\\
k_{Y_1}^{-1}:~C_1,P_1,P_2,\ldots ,P_{n-3}&\mapsto&ip_1\in P_n,
\end{eqnarray*}
and $k_{Y_1}:P_n\longleftrightarrow P_{n-2}$ with
\begin{eqnarray*}
P_n\ni u&\mapsto&\frac{1}{-a_n^2u-(n-1)a_n+a_{n-1}}\in P_{n-2},\\
P_{n-2}\ni u&\mapsto&-\frac{1}{a_n^2u}-\frac{a_{n-1}}{a_n}\in P_n.
\end{eqnarray*}
We will prove the following result
\begin{theorem}
Let $n=deg(F)\geq 3$ be odd. Let $Y$ be the manifold constructed in Section 4 in \cite{bedford-kim-tuyen-abarenkova-maillard}. Let $L_j:\mathbb
C^{n+1}\rightarrow\mathbb C$ ($j=0,1,\ldots ,n$) be linear functions defined in Section 3. Then there exists $0\leq j\leq n-2$ such that
\begin{eqnarray*}
L_i(a_0,a_1,\ldots ,a_n)=0
\end{eqnarray*}
for all $0\leq i\leq j$. Moreover, we choose $j$ as large as possible, that is either $j=n-2$ or $L_{j+1}(a_0,a_1,\ldots ,a_n)\not= 0$.

If $a_0=\frac{2}{m+1}$ for some $m=0,1,2,\ldots $, let $\widehat{Y}$ denote the manifold constructed by blowing up the manifold $Y$ at the points
$r_0,q_0,r_1,q_1,\ldots ,r_m,q_m$ as in the proof of Theorem \ref{DegreeEvenCase}. Otherwise, i.e. if $a_0\not=\frac{2}{m+1}$ for any $m=0,1,2,\ldots $, let $\widehat{Y}$
denote the space $Y$ itself.

We have one of the following alternatives

Case 1: $j<n-2$. Then there exists spaces $Y_1=\widehat{Y},Y_2,\ldots ,Y_{j},Y_{j+1}=Z$, where $Y_{i+1}\rightarrow Y_{i}$ is a one point-blowup for
$i\leq j-1$, such that the induced map $k_{Z}:Z\rightarrow Z$ is A.S.

Case 2: $j=n-2$.

Subcase 2.1: $a_0\not= \frac{n+1}{2}+\frac{m}{2(1+m)}$ for any $m=0,1,2,\ldots $. Then there exists spaces $Y_1=\widehat{Y},Y_2,\ldots ,Y_{n}=Z$, where
$Y_{i+1}\rightarrow Y_{i}$ is a one point-blowup for $i\leq n-1$, such that the induced map $k_{Z}:Z\rightarrow Z$ is A.S.

Case 2.2: $a_0=\frac{n+1}{2}+\frac{m}{2(1+m)}$ for some $m=0,1,2,\ldots $. Let $Y_n$ be the space in Subcase 2.1. Then there exists spaces
$Y_n,Y_{n+1},\ldots ,Y_{n+2m+2}=Z$, where $Y_{n+i+1}\rightarrow Y_{n+i}$ is a one point-blowup for $i\leq 2m+1$, such that the induced map
$k_{Z}:Z\rightarrow Z$ is A.S.\label{OddCase}\end{theorem}
The proof of Theorem \ref{OddCase} is divided into some steps.
\begin{lemma}
Let $1\leq m\leq n-3$ be an integer. Assume that
\begin{equation}
L_j(a_0,a_1,\ldots ,a_n)=0\label{FirstBlowups.1}
\end{equation}
for all $j=0,1,2,\ldots ,m$, and
\begin{eqnarray*}
L_{m+1}(a_0,a_1,\ldots ,a_n)\not= 0.
\end{eqnarray*}
Construct a sequence of spaces $Y_j$ ($2\leq j\leq m+1$) by induction as follows: $Y_j$ is the blowup of $Y_{j-1}$ at a point $ip_{j-1}=ep_{j-1}\in
P_{n-1+j-1}$, and $P_{n-1+j}$ is the exceptional fiber of the blowup $Y_j\rightarrow Y_{j-1}$. Here the points $ip_j$ and $ep_j$ ($1\leq j\leq m+1$) are
defined as
\begin{eqnarray*}
ip_j&=&\frac{c_j+d_j}{a_n^2},~ep_j=\frac{\gamma _j+c_j}{a_n^2},
\end{eqnarray*}
where $-a_n^2u+c_j$ is the coefficient of $s^j$ of the Taylor expansion of the function
\begin{eqnarray*}
\frac{1+s}{ip_0+ip_1s+\ldots +ip_{j-1}s^{j-1}+s^ju}
\end{eqnarray*}
near $s=0$, $d_j$ is the coefficient of $s^j$ in the polynomial $s^nF(-1-\frac{1}{s})$, $\gamma _j$ is the coefficient of $s^j$ in the polynomial
$-(1+s)s^nF(\frac{1}{s})$. Moreover for small values of $s$:
\begin{eqnarray*}
\frac{1+s}{ip_0+ip_1s+\ldots +ip_{j-1}s^{j-1}+s^ju}+s^nF(-1-\frac{1}{s})&=&(-a_n^2u+c_j)s^j+d_js^j+O(s^{j+1}),\\
\frac{1+s}{ip_0+ip_1s+\ldots +ip_{j-1}s^{j-1}+s^ju}-(1+s)s^nF(\frac{1}{s})&=&(-a_n^2u+c_j)s^j+\gamma _js^j+O(s^{j+1}).
\end{eqnarray*}
Then the induced map $k_{Y_{m+1}}:Y_{m+1}\rightarrow Y_{m+1}$ is A.S., and acts as follows:
\begin{eqnarray*}
k_{Y_{m+1}}:~C_1,C_2,P_1,\ldots ,P_{n-1-(m+1)-1}&\mapsto&ep _{m+1}\in P_{n-1+m+1},\\
k_{Y_{m+1}}^{-1}:C_1,P_1,\ldots ,P_{n-1-(m+1)-1}&\mapsto&ip_{m+1}\in P_{n-1+m+1},
\end{eqnarray*}
and $k_{Y_{m+1}}:P_{n-1+m+1}\longleftrightarrow P_{n-1-(m+1)}$ is
\begin{eqnarray*}
k_{Y_{m+1}}:~P_{n-1+m+1}\ni u&\mapsto&\frac{(-1)^{n-(m+1)}}{-a_n^2u+d_{m+1}+c_{m+1}}\in P_{n-1-(m+1)},\\
k_{Y_{m+1}}:~P_{n-1-(m+1)}\ni u&\mapsto&\frac{(-1)^{n-(m+1)}}{-a_n^2u}+\frac{c_{m+1}+\gamma _{m+1}}{a_n^2}\in P_{n-1+(m+1)}.
\end{eqnarray*}
In these formulas we choose the coordinate projection at $P_{n-1+j}$ as
\begin{eqnarray*}
(s,u)\mapsto [s^n(ip_0+ip_1s+\ldots +ip_{j-1}s^{j-1}+s^ju):1:s^{n-1}(ip_0+ip_1s+\ldots +ip_{j-1}s^{j-1}+s^ju)].
\end{eqnarray*}
\label{FirstBlowups}\end{lemma}
\begin{proof}
1) Step 1: We will prove by induction on $m$ that if (\ref{FirstBlowups.1}) is satisfied for $0\leq j\leq m$ then a sequence of spaces $Y_j$ ($1\leq
j\leq m+1$) exists and satisfies all the conclusion of Lemma \ref{FirstBlowups} except the conclusion that $k_{Y_{m+1}}$ is A.S.

Proof: Assume by induction that spaces $Y_0,Y_1,\ldots ,Y_l$ ($l<m+1$) was constructed with the following properties: $Y_j$ is the blowup of $Y_{j-1}$ at
a point $ip_{j-1}=ep_{j-1}\in P_{n-1+j-1}$, and $P_{n-1+j}$ is the exceptional fiber of the blowup $Y_j\rightarrow Y_{j-1}$. Here the points $ip_j$ and
$ep_j$ ($1\leq j\leq l$) are defined as
\begin{eqnarray*}
ip_j&=&\frac{c_j+d_j}{a_n^2},~ep_j=\frac{\gamma _j+c_j}{a_n^2},
\end{eqnarray*}
where $-a_n^2u+c_j$ is the coefficient of $s^j$ of the Taylor expansion of the function
\begin{eqnarray*}
\frac{1+s}{ip_0+ip_1s+\ldots +ip_{j-1}s^{j-1}+s^ju}
\end{eqnarray*}
near $s=0$, $d_j$ is the coefficient of $s^j$ in the polynomial $s^nF(-1-\frac{1}{s})$, $\gamma _j$ is the coefficient of $s^j$ in the polynomial
$-(1+s)s^nF(\frac{1}{s})$. Moreover for small values of $s$:
\begin{eqnarray*}
\frac{1+s}{ip_0+ip_1s+\ldots +ip_{j-1}s^{j-1}+s^ju}+s^nF(-1-\frac{1}{s})&=&(-a_n^2u+c_j)s^j+d_js^j+O(s^{j+1}),\\
\frac{1+s}{ip_0+ip_1s+\ldots +ip_{j-1}s^{j-1}+s^ju}-(1+s)s^nF(\frac{1}{s})&=&(-a_n^2u+c_j)s^j+\gamma _js^j+O(s^{j+1}).
\end{eqnarray*}
The induced map $k_{Y_{l}}:Y_{l}\rightarrow Y_{l}$ acts as follows:
\begin{eqnarray*}
k_{Y_{l}}:~C_1,C_2,P_1,\ldots ,P_{n-1-l-1}&\mapsto&ep _{l}\in P_{n-1+l},\\
k_{Y_{l}}^{-1}:C_1,P_1,\ldots ,P_{n-1-l-1}&\mapsto&ip_{l}\in P_{n-1+l},
\end{eqnarray*}
and $k_{Y_{l}}:P_{n-1+l}\longleftrightarrow P_{n-1-l}$ is
\begin{eqnarray*}
k_{Y_{l}}:~P_{n-1+l}\ni u&\mapsto&\frac{(-1)^{n-l}}{-a_n^2u+d_{l}+c_{l}}\in P_{n-1-l},\\
k_{Y_{l}}:~P_{n-1-l}\ni u&\mapsto&\frac{(-1)^{n-l}}{-a_n^2u}+\frac{c_{l}+\gamma _{l}}{a_n^2}\in P_{n-1+l}.
\end{eqnarray*}

The starting point $l=0$ can be easily checked to satisfy the above conditions.

i) Claim 1: The following two facts are equivalent

Fact 1: $ep_j=ip_j$ for all $0\leq j\leq l$.

Fact 2: $L_j(a_0,a_1,\ldots ,a_n)=0$ for all $0\leq j\leq l$.

Proof of Claim 1: From the definition of $ep_j$ and $ip _j$ it is not much difficult to check Claim 1 (the reader may check with concrete examples to see
how this works).

ii) Claim 2: $k_{Y_l}$ is not A.S.

Proof of Claim 2: From Claim 1 and  the action of $k_{Y_l}$ we see that
\begin{eqnarray*}
k_{Y_{l}}:~C_1,C_2,P_1,\ldots ,P_{n-1-l-1}\mapsto ep _{l}=ip _l\in P_{n-1+l},
\end{eqnarray*}
which is an indeterminate point of $k_{Y_l}$. Hence $k_{Y_l}$ is not A.S.

iii) Claim 3: Let $Y_{l+1}$  be the blowup of $Y_l$ at the point $ep_l=ip_l\in P_{n-1+l}$, and let $P_{n-1+l+1}$ be the exceptional fiber of this blowup.
Choose the coordinate projection at $P_{n-1+l+1}$ as described in the statement of Lemma \ref{FirstBlowups}. Then the action of the induced map
$k_{Y_{l+1}}:Y_{l+1}\rightarrow Y_{l+1}$ is
\begin{eqnarray*}
k_{Y_{l+1}}:~C_1,C_2,P_1,\ldots ,P_{n-1-(l+1)-1}&\mapsto&ep _{l+1}\in P_{n-1+l+1},\\
k_{Y_{l+1}}^{-1}:C_1,P_1,\ldots ,P_{n-1-(l+1)-1}&\mapsto&ip_{l+1}\in P_{n-1+l+1},
\end{eqnarray*}
and $k_{Y_{l+1}}:P_{n-1+l+1}\longleftrightarrow P_{n-1-(l+1)}$ is
\begin{eqnarray*}
k_{Y_{l+1}}:~P_{n-1+l+1}\ni u&\mapsto&\frac{(-1)^{n-(l+1)}}{-a_n^2u+d_{l+1}+c_{l+1}}\in P_{n-1-(l+1)},\\
k_{Y_{l+1}}:~P_{n-1-(l+1)}\ni u&\mapsto&\frac{(-1)^{n-(l+1)}}{-a_n^2u}+\frac{c_{l+1}+\gamma _{l+1}}{a_n^2}\in P_{n-1+(l+1)}.
\end{eqnarray*}

Proof of Claim 3: First we compute the image of a generic point $u\in P_{n-1+l+1}$ under the map $k_{Y_{l+1}}$. Use the formula
\begin{eqnarray*}
k_{Y_{l+1}}[u]=\lim _{s\rightarrow 0}k[s^n(ip_0+ip_1s+\ldots +ip_{l}s^{l}+s^{l+1}u):1:s^{n-1}(ip_0+ip_1s+\ldots +ip_{l}s^{l}+s^{l+1}u)],
\end{eqnarray*}
and the fact that
\begin{eqnarray*}
\frac{1+s}{ip_0+ip_1s+\ldots +ip_{j-1}s^{j-1}+s^ju}+s^nF(-1-\frac{1}{s})&=&(-a_n^2u+c_j)s^j+d_js^j+O(s^{j+1}),\\
\frac{1+s}{ip_0+ip_1s+\ldots +ip_{j-1}s^{j-1}+s^ju}-(1+s)s^nF(\frac{1}{s})&=&(-a_n^2u+c_j)s^j+\gamma _js^j+O(s^{j+1}).
\end{eqnarray*}
for all $j\leq l+1$, it is not hard to see that
\begin{eqnarray*}
k_{Y_{l+1}}:~P_{n-1+l+1}\ni u&\mapsto&\frac{(-1)^{n-(l+1)}}{-a_n^2u+d_{l+1}+c_{l+1}}\in P_{n-1-(l+1)}.
\end{eqnarray*}
Now we compute the image of a generic point $u\in P_{n-1+(l+1)}$ under the map $k_{Y_{l+1}}^{-1}$. Use the formula
\begin{eqnarray*}
k_{Y_{l+1}}^{-1}[u]=\lim _{s\rightarrow 0}k^{-1}[s^n(ip_0+ip_1s+\ldots +ip_{l}s^{l}+s^{l+1}u):1:s^{n-1}(ip_0+ip_1s+\ldots +ip_{l}s^{l}+s^{l+1}u)],
\end{eqnarray*}
and arguing as above, it is not hard to see that
\begin{eqnarray*}
k_{Y_{l+1}}^{-1}:~P_{n-1+l+1}\ni u&\mapsto&\frac{(-1)^{n-(l+1)}}{-a_n^2u+\gamma_{l+1}+c_{l+1}}\in P_{n-1-(l+1)}.
\end{eqnarray*}
Then the image of a generic point $u\in P_{n-1-(l+1)}$ using $k_{Y_{l+1}}=(k_{Y_{l+1}}^{-1})^{-1}$ can be computed as follows
\begin{eqnarray*}
k_{Y_{l+1}}:~P_{n-1-(l+1)}\ni u\mapsto \frac{(-1)^{n-(l+1)}}{-a_n^2u}+\frac{\gamma _{l+1}+c_{l+1}}{a_n^2}\in P_{n-1+l+1}.
\end{eqnarray*}
(The reason why we computed $k_{Y_{l+1}}:P_{n-1-(l+1)}\rightarrow P_{n-1+(l+1)}$ via $k_{Y_{l+1}}^{-1}:P_{n-1+(l+1)}\rightarrow P_{n-1-(l+1)}$ is because
the formula for coordinate projection of $P_{n-1-(l+1)}$ is much more simpler than that of $P_{n-1+(l+1)}$.) Hence
$k_{Y_{l+1}}:P_{n-1+l+1}\longleftrightarrow P_{n-1-(l+1)}$. This fact, and the inductional assumption that
\begin{eqnarray*}
k_{Y_{l}}:~C_1,C_2,P_1,\ldots ,P_{n-1-l-1}&\mapsto&ep _{l}\in P_{n-1+l},
\end{eqnarray*}
imply that $C_1,C_2,P_1,\ldots ,P_{n-1-(l+1)-1}$ are exceptional curves for $k_{Y_{l+1}}$, and hence their images must be the points lie in
$P_{n-1+(l+1)}$ which is indeterminate points for $k_{Y_{l+1}}^{-1}$. From the formula for $k_{Y_{l+1}}^{-1}:P_{n-1+(l+1)}\rightarrow P_{n-1-(l+1)}$
which we found above, there is only such a point which is exactly $ep_{l+1}$. Hence
\begin{eqnarray*}
k_{Y_{l+1}}:~C_1,C_2,P_1,\ldots ,P_{n-1-(l+1)-1}&\mapsto&ep _{l+1}\in P_{n-1+l+1}.
\end{eqnarray*}
Similarly
\begin{eqnarray*}
k_{Y_{l+1}}^{-1}:C_1,P_1,\ldots ,P_{n-1-(l+1)-1}&\mapsto&ip_{l+1}\in P_{n-1+l+1}.
\end{eqnarray*}

Using the above Claims we complete the proof of Step 1.

2) Step 2: Completion of the proof of Lemma \ref{FirstBlowups}: In Step 1 we showed that a sequence of spaces $Y_j$ ($1\leq j\leq m+1$) exists and
satisfies all the conclusion of Lemma \ref{FirstBlowups} except the conclusion that $k_{Y_{m+1}}$ is A.S. Now we show that $k_{Y_{m+1}}$ is A.S.

From Step 1 we have
\begin{eqnarray*}
k_{Y_{m+1}}^2:P_{n-1+m+1}\ni u\mapsto u+ep_{m+1}-ip_{m+1}\in P_{n-1+m+1}.
\end{eqnarray*}
Hence the orbit of the exceptional curves are
\begin{eqnarray*}
k_{Y_{m+1}}^{2l+1}:~C_1,C_2,P_1,\ldots ,P_{n-1-(m+1)-1}\mapsto ep _{m+1}+l(ep_{m+1}-ip_{m+1})\in P_{n-1+m+1}
\end{eqnarray*}
never land on the indeterminate point $ip_{m+1}\in P_{n-1+m+1}$, since $ep_{m+1}\not= ip_{m+1}$ as can be easily seen from Claim 1 in Step 1 and our
assumption that $L_{m+1}(a_0,a_1,\ldots ,a_n)\not= 0$. This implies that $k_{Y_{m+1}}$ is A.S.
\end{proof}
\begin{lemma}
Assume that
\begin{equation}
L_j(a_0,a_1,\ldots ,a_n)=0\label{FirstBlowups.1}
\end{equation}
for all $j=0,1,2,\ldots ,n-3$. Construct the spaces $Y_1,\ldots ,Y_{n-2}$ as described in Lemma \ref{FirstBlowups}.

Case 1: $L_{n-2}(a_0,a_1,\ldots ,a_n)\not= 0$. Then the induced map $k_{Y_{n-2}}$ is A.S.

Case 2: $L_{n-2}(a_0,a_1,\ldots ,a_n)=0$. Then the induced map $k_{Y_{n-2}}$ is not A.S, and $ep_{n-2}=ip_{n-2}$. Construct the space $Y_{n-1}$ as the
blowup of $Y_{n-1}$ at the point $ep_{n-2}$, and call $P_{n-1+n-1}$ the exceptional fiber of this blowup $Y_{n-1}$. Then the action of the induced map
$k_{Y_{n-1}}$ is
\begin{eqnarray*}
k_{Y_{n-1}}:~P_{n-1+n-1}&\longleftrightarrow&C_1,~C_2\mapsto ep_{n-1}\mapsto [0:0:1]=e_2,
\end{eqnarray*}
where $ep_{n-1}\in P_{n-1+n-1}$ is constructed in the same way as $ep_1,\ldots ,ep_{n-2}$. The map $k_{Y_{n-1}}$ has no indeterminate point lying in
$P_{n-1+n-1}$, but it is not A.S.

Let $Y_n$ be the blowup of $Y_{n-1}$ at two points $ep_{n-1}\in P_{n-1+n-1}$ and $e_2=[0:0:1]$, call $P_{n-1+n}$ and $E_2$ the exceptional fibers of this
blowup $Y_n\rightarrow Y_{n-1}$. Let the coordinate projection at $P_{n-1+n}$ as
\begin{eqnarray*}
(s,u)\mapsto [s^n(ep_0+ep_1s+\ldots +ep_{n-1}s^{n-1}+s^nu):1:s^{n-1}(ep_0+ep_1s+\ldots +ep_{n-1}s^{n-1}+s^nu)],
\end{eqnarray*}
and the coordinate projection at $E_{2}$ is
\begin{eqnarray*}
(s,u)\mapsto [s:su:1].
\end{eqnarray*}
(Recall that we do not have a point $ip_{n-1}$, however we do have the points $ip_0=ep_0,ip_1=ep_1,\ldots ,ip_{n-2}=ep_{n-2}$.) Under these coordinates
then the induced map $k_{Y_n}:Y_n\rightarrow Y_n$ is
\begin{eqnarray*}
k_{Y_n}:C_2\mapsto ep_n\in P_{n-1+n},
\end{eqnarray*}
and $k_{Y_n}:P_{n-1+n}\longleftrightarrow E_2$ as
\begin{eqnarray*}
k_{Y_n}:~P_{n-1+n}\ni u&\mapsto&-a_n^2u+c_n+d_n-(n+1)\in E_2,\\
k_{Y_n}:~E_2\ni u&\mapsto&\frac{-u+c_n+\gamma _n+1}{a_n^2} \in P_{n-1+n}.
\end{eqnarray*}
Here the constants $ep_n,c_n,d_n,\gamma _n$ are
\begin{eqnarray*}
ep_n=\frac{c_n+\gamma _n}{a_n^2},
\end{eqnarray*}
and $d_n$ is the coefficient of $s^n$ in the polynomial $s^nF(-1-\frac{1}{s})$, $\gamma _n$ is the coefficient of $s^n$ in the polynomial
$(1+s)s^nF(\frac{1}{s})$, and when using Taylor's expansion for $s$ small enough
\begin{eqnarray*}
\frac{1+s}{ep_0+ep_1s+\ldots +ep_{n-1}s^{n-1}+s^nu}+s^nF(-1-\frac{1}{s})&=&(-a_n^2u+c_n+d_n)s^n+O(s^{n+1}).\\
\frac{1+s}{ep_0+ep_1s+\ldots +ep_{n-1}s^{n-1}+s^nu}-(1+s)s^nF(-1-\frac{1}{s})&=&(-a_n^2u+c_n+\gamma _n)s^n+O(s^{n+1}).
\end{eqnarray*}
Moreover $0\in E_2$ is the only indeterminate point lying in $E_2$ of $k_{Y_n}$.

Subcase 2.1: $a_0\not= \frac{n+1}{2}+\frac{l}{2(1+l)}$ for any $l=0,1,2,\ldots $. Then the induced map $k_{Y_{n}}:Y_n\rightarrow Y_n$ is A.S.

Subcase 2.2: $a_0=\frac{n+1}{2}+\frac{l}{2(1+l)}$ for some $l=0,1,2,\ldots $. Let $Z$ be the space constructed by blowing up $Y_n$ at points
\begin{eqnarray*}
ep_n&\in &P_{n-1+n},~k_{Y_n}(ep_n)\in E_2,\\
k_{Y_n}^2(ep_n)&\in&P_{n-1+n},\\
\ldots&&\\
k_{Y_n}^{2l}(ep_n)&\in&P_{n-1+n},~k_{Y_n}^{2l+1}(ep_n)=0\in E_2.
\end{eqnarray*}
Then the induced map $k_Z:Z\rightarrow Z$ is A.S.
\label{HigherBlowup}\end{lemma}
\begin{proof}
Case 1 and the action of the induced map $k_{Y_{n-1}}\rightarrow k_{Y_{n-1}}$ can be proved as in Lemma \ref{FirstBlowups}.

Now the action
\begin{equation}
k_{Y_n}:~P_{n-1+n}\ni u\mapsto -a_n^2u+c_n+d_n-(n+1)\in E_2\label{HigherBlowup.1}
\end{equation}
can be computed as in Lemma \ref{FirstBlowups}. In the same way we can compute
\begin{eqnarray*}
k_{Y_n}^{-1}:~P_{n-1+n}\ni u \mapsto -a_n^2u+c_n+\gamma _n-1\in E_2,
\end{eqnarray*}
hence using $k_{Y_n}=(k_{Y_n}^{-1})^{-1}$ we get
\begin{equation}
k_{Y_n}:~E_2\ni u\mapsto\frac{-u+c_n+\gamma _n+1}{a_n^2} \in P_{n-1+n}.\label{HigherBlowup.2}
\end{equation}
That $0\in E_2$ is the only indeterminate point lying in $E_2$ of $k_{Y_n}$ is not hard to see, and that the image of $C_2$  is a point $ep_n\in
P_{n-1+n}$ can be proved by the same argument in the proof of Lemma \ref{FirstBlowups}. Now we compute $ep_n$. We have $C_2\cap E_2=1\in E_2$ which is
not an indeterminate point of $k_{Y_n}$, hence using (\ref{HigherBlowup.2})
\begin{eqnarray*}
ep_n=k_{Y_n}(C_2)=k_{Y_n}([1]_{E_2})=\frac{c_n+\gamma _n}{a_n^2}.
\end{eqnarray*}
From (\ref{HigherBlowup.1}) and (\ref{HigherBlowup.2}) we get
\begin{eqnarray*}
k_{Y_n}^2:~E_2\ni u\mapsto u-\gamma _n+d_n-(n+2).
\end{eqnarray*}
Then
\begin{eqnarray*}
k_{Y_n}(ep_n)=-\gamma _n+d_n-(n+1)\in E_2.
\end{eqnarray*}
Using the formulas of $\gamma _n$ and $d_n$, and using Lemma \ref{problem2} we have $d_n-\gamma _n=2a_0$. Hence the orbit of $C_2$ is
\begin{eqnarray*}
k_{Y_n}^{2l+2}:C_2\mapsto 2a_0(l+1)-(n+1)(l+1)-l\in E_2.
\end{eqnarray*}
This orbit lands on the indeterminate point $0\in E_2$ iff
\begin{eqnarray*}
2a_0(l+1)-(n+1)(l+1)-l=0
\end{eqnarray*}
for some $l=0,1,2,\ldots$. Then Case 2 easy follows.
\end{proof}
\section{Proof of Theorem \ref{DegreeOddCase}}\label{sectionProofTheoremDegreeOddCase}
In this section we prove Theorem \ref{DegreeOddCase}. Let $Z$ be the spaces constructed in Theorem \ref{OddCase}. Since the map $k_Z:Z\rightarrow Z$ is A.S.,
we obtain $\delta (k_F)$ as the spectral radius of $k_Z^*$.

1. Case $n=deg (F)$ is odd; $a_0\not= 2/(m+1)$ for any $m=0,1,2,\ldots $; $L_i(a_0,a_1,\ldots ,a_n)=0$ for any $0\leq i\leq h$, $L_{h+1}(a_0,\ldots
,a_n)\not= 0$ where $0\leq h<n-2$. As noted before, in this case $h$ must be an even integer.
\begin{lemma}
The spectral radius of $k_Z^*:Pic(Z)\rightarrow Pic(Z)$ is the largest real root of the polynomial $x^3-nx^2-(n+1-h)x-1.$
\label{CharacteristicPolynomialNOdd1}\end{lemma}
\begin{proof}
Let $Z$ be the space constructed in Theorem \ref{OddCase}. Let $H_Z,E_1,Q,P_1,\ldots ,P_{n-1+h+1}$ be a basis for $Pic(Z)$ (see convenience in the proof
of Theorem \ref{DegreeEvenCase}). In this basis then
\begin{eqnarray*}
C_1&=&H_Z-E_1-Q-\sum _{j=1}^{n-1}(j+1)P_j-n\sum _{j=1}^{h+1}P_{n-1+j},\\
C_2&=&H_Z,~C_3=H_Z-E_1-Q-\sum _{j=1}^{n-1}jP_j-(n-1)\sum _{j=1}^{h+1}P_{n-1+j},\\
C_4&=&2H_Z-E_1-2Q-\sum _{j=1}^{n-1}jP_j-(n-1)\sum _{j=1}^{h+1}P_{n-1+j}.
\end{eqnarray*}
As in the proof of Theorem \ref{DegreeEvenCase}, $k_Z^*:Pic(Z)\rightarrow Pic(Z)$ acts as
\begin{eqnarray*}
k_Z^*(H_Z)&=&(2n+1)H_Z-nE_1-(n+1)Q-(n+1)\sum _{j=1}^{n-1}jP_j-\sum _{j=1}^{h+1}(n^2-1+j)P_{n-1+j},\\
k_Z^*(E_1)&=&E_1,~k_Z^*(Q)=H_Z-E_1-Q-\sum _{j=1}^{n-1}jP_j-(n-1)\sum _{j=1}^{h+1}P_{n-1+j},\\
k_Z^*(P_{n-1-j})&=&0,~j>h+1,~k_Z^*(P_{n-1-(h+1)})=P_{n-1+h+1},\\
k_Z^*(P_{n-1-j})&=&P_{n-1+j},~j=0,\ldots ,h,~k_Z^*(P_{n-1+j})=P_{n-1-j},~j=0,\ldots ,h,\\
k_Z^*(P_{n-1+h+1})&=&C_1+C_2+P_1+\ldots +P_{n-1-(h+1)}\\
&=& 2H_Z-\sum _{j=1}^{n-1-(h+1)}jP_j-\sum _{j=n-1-h}^{n-1}(j+1)P_j-n\sum _{j=1}^{h+1}P_{n-1+j}.
\end{eqnarray*}
The spectral radius of $k_Z^*$ can be computed as the greatest real zero of the characteristic polynomial of the matrix representation of $k_Z^*$ restricted to
$\{H_Z,Q,P_{n-1-h-1},P_{n-1+h+1}\}$ which is
$$M_1=\left (\begin{array}{llll}
2n+1&-(n+1)&-(n+1)(n-1-h-1)&-(n^2-1+h+1)\\
1&-1&-(n-1-h-1)&-(n-1)\\
0&0&0&1\\
2&-1&-(n-1-h-1)&-n
\end{array}\right )$$
The characteristic polynomial $P(x)=det(M_1-xI)$ of $M_1$ is $$P(x)=-x(x^3-nx^2-(n+1-h)x-1).$$ From this the conclusions of Lemma \ref{CharacteristicPolynomialNOdd1} follow.
\end{proof}

2. Case $n=deg (F)$ is odd; $a_0= 2/(m+1)$ for some $m=0,1,2,\ldots $; $L_i(a_0,a_1,\ldots ,a_n)=0$ for any $0\leq i\leq h$, $L_{h+1}(a_0,\ldots
,a_n)\not= 0$ where $0\leq h<n-2$. As noted before, in this case $h$ must be an even integer.
\begin{lemma}
The spectral radius of $k_Z^*:Pic(Z)\rightarrow Pic(Z)$ is the largest real root of the polynomial $x^{2m+1}(x^3-nx^2-(n-h+1)x-1)+x^3+x^2+nx+n-h-1.$
\label{CharacteristicPolynomialNOdd2}\end{lemma}
\begin{proof}
Let $Z$ be the space constructed in Theorem \ref{OddCase}. Let $H_Z,E_1,Q,P_1,\ldots ,P_{n-1+h+1}$, $Q_1,\ldots ,Q_m,R_0,R_1,\ldots ,R_m$ be a basis for
$Pic(Z)$. In this basis then
\begin{eqnarray*}
C_1&=&H_Z-E_1-Q-\sum _{j=1}^{n-1}(j+1)P_j-n\sum _{j=1}^{h+1}P_{n-1+j}-\sum _{j=1}^mQ_j,\\
C_2&=&H_Z-R_m,\\
C_3&=&H_Z-E_1-Q-\sum _{j=1}^{n-1}jP_j-(n-1)\sum _{j=1}^{h+1}P_{n-1+j}-\sum _{j=1}^mQ_j-\sum _{j=0}^mR_j,\\
C_4&=&2H_Z-E_1-2Q-\sum _{j=1}^{n-1}jP_j-(n-1)\sum _{j=1}^{h+1}P_{n-1+j}-2\sum _{j=1}^mQ_j-R_m.
\end{eqnarray*}
Then $k_Z^*:Pic(Z)\rightarrow Pic(Z)$ acts as
\begin{eqnarray*}
k_Z^*(H_Z)&=&(2n+1)H_Z-nE_1-(n+1)Q-(n+1)\sum _{j=1}^{n-1}jP_j-\sum _{j=1}^{h+1}(n^2-1+j)P_{n-1+j}\\
&&-(n+1)\sum _{j=1}^mQ_j-(n+1)R_m,\\
k_Z^*(E_1)&=&E_1,\\
k_Z^*(Q)&=&H_Z-E_1-Q-\sum _{j=1}^{n-1}jP_j-(n-1)\sum _{j=1}^{h+1}P_{n-1+j}-\sum _{j=1}^mQ_j-\sum _{j=0}^mR_j,\\
k_Z^*(P_{n-1-j})&=&0,~j>h+1,~k_Z^*(P_{n-1-(h+1)})=P_{n-1+h+1},\\
k_Z^*(P_{n-1-j})&=&P_{n-1+j},~j=0,\ldots ,h,~k_Z^*(P_{n-1+j})=P_{n-1-j},~j=0,\ldots ,h,\\
k_Z^*(P_{n-1+h+1})&=&2H_Z-\sum _{j=1}^{n-1-(h+1)}jP_j-\sum _{j=n-1-h}^{n-1}(j+1)P_j-n\sum _{j=1}^{h+1}P_{n-1+j}-\sum _{j=1}^mQ_j-R_m,\\
k_Z^*(R_0)&=&2H_Z-E_1-2Q-\sum _{j=1}^{n-1}jP_j-(n-1)\sum _{j=1}^{h+1}P_{n-1+j}-2\sum _{j=1}^mQ_j-R_m,\\
k_Z^*(R_j)&=&Q_j,~1\leq j\leq m,~k_Z^*(Q_j)=R_{j-1},~1\leq j\leq m.
\end{eqnarray*}
The spectral radius of $k_Z^*$ can be computed as the greatest real zero of the characteristic polynomial of the matrix representation $M_1$ of $k_Z^*$ restricted
 to $\{H_Z,Q,P_{n-1-h-1},P_{n-1+h+1}$, $Q_1,\ldots ,Q_m$, $R_0,R_1,\ldots ,R_m\}$ which is $$P(x)=-x[x^{2m+1}(x^3-nx^2-(n-h+1)x-1)+x^3+x^2+nx+n-h-1].$$
From this the conclusions of Lemma \ref{CharacteristicPolynomialNOdd2} follow.
\end{proof}

3. Case $n=deg (F)$ is odd; $a_0\not= 2/(m+1)$ for any $m=0,1,2,\ldots $; $a_0\not= \frac{n+1}{2}+\frac{l}{2(l+1)}$ for any $l=0,1,2,\ldots $;
$L_i(a_0,a_1,\ldots ,a_n)=0$ for any $0\leq i\leq n-2$.
\begin{lemma}
The spectral radius of $k_Z^*:Pic(Z)\rightarrow Pic(Z)$ is the largest real root of the polynomial $x^3-nx^2-2x-1.$
\label{CharacteristicPolynomialNOdd3}\end{lemma}
\begin{proof}
Let $Z$ be the space constructed in Theorem \ref{OddCase}. Let $H_Z,E_1,Q,P_1,\ldots ,P_{n-1+n},E_2$ be a basis for $Pic(Z)$. Then  $k_Z^*:Pic(Z)\rightarrow Pic(Z)$ acts
as
\begin{eqnarray*}
k_Z^*(H_Z)&=&(2n+1)H_Z-nE_1-(n+1)Q-(n+1)\sum _{j=1}^{n-1}jP_j\\
&&-\sum _{j=1}^{n-1}(n^2-1+j)P_{n-1+j}-(n^2+n-2)P_{n-1+n}-nE_2,\\
k_Z^*(E_1)&=&E_1,~k_Z^*(Q)=H_Z-E_1-Q-\sum _{j=1}^{n-1}jP_j-(n-1)\sum _{j=1}^{n}P_{n-1+j},\\
k_Z^*(P_{n-1-j})&=&P_{n-1+j},~j=0,\ldots ,n-2,~k_Z^*(P_{n-1+j})=P_{n-1-j},~j=0,\ldots ,n-2,\\
k_Z^*(P_{n-1+n-1})&=&H_Z-E_1-Q-\sum _{j=1}^{n-1}(j+1)P_j-n\sum _{j=1}^nP_{n-1+j}-E_2,\\
k_Z^*(P_{n-1+n})&=&E_2+C_2=H_Z,~k_Z^*(E_2)=P_{n-1+n}.
\end{eqnarray*}
The spectral radius of $k_Z^*$ can be computed as the greatest real zero of the characteristic polynomial of the matrix representation $M_1$ of $k_Z^*$ restricted
to $\{H_Z,Q,P_{n-1+n-1},P_{n-1+n},E_2\}$, which is
$$M_1=\left ( \begin{array}{lllll}
2n+1&-(n+1)&-(n^2+n-2)&-(n^2+n-2)&-n\\
1&-1&-(n-1)&-(n-1)&0\\
1&-1&-n&-n&-1\\
1&0&0&0&0\\
0&0&0&1&0
\end{array}\right )$$
. The characteristic polynomial $P(x)=det(M_1-xI)$ of $M_1$ is $$P(x)=-(x-1)(x+1)(x^3-nx^2-2x-1).$$ From this the conclusions of Lemma \ref{CharacteristicPolynomialNOdd3} follow.
\end{proof}

4. Case $n=deg (F)$ is odd; $a_0= 2/(m+1)$ for some $m=0,1,2,\ldots $; $a_0\not= \frac{n+1}{2}+\frac{l}{2(l+1)}$ for any $l=0,1,2,\ldots $;
$L_i(a_0,a_1,\ldots ,a_n)=0$ for any $0\leq i\leq n-2$.
\begin{lemma}
The spectral radius of $k_Z^*:Pic(Z)\rightarrow Pic(Z)$ is the largest real root of the polynomial $x^{2m}(x^3-nx^2-2x-1)+x^2+x+n.$
\label{CharacteristicPolynomialNOdd4}\end{lemma}
\begin{proof}
Let $Z$ be the space constructed in Theorem \ref{OddCase}. Let $H_Z,E_1,Q,P_1,\ldots ,P_{n-1+n},E_2$, $R_0,R_1,\ldots ,R_m,Q_1,\ldots ,Q_m$ be a basis
for $Pic(Z)$. Then $k_Z^*:Pic(Z)\rightarrow Pic(Z)$ acts as
\begin{eqnarray*}
k_Z^*(H_Z)&=&(2n+1)H_Z-nE_1-(n+1)Q-(n+1)\sum _{j=1}^{n-1}jP_j-\sum _{j=1}^{n-1}(n^2-1+j)P_{n-1+j}\\
&&-(n^2+n-2)P_{n-1+n}-nE_2-(n+1)\sum _{j=1}^{m}Q_j-(n+1)R_m,\\
k_Z^*(E_1)&=&E_1,~k_Z^*(Q)=H_Z-E_1-Q-\sum _{j=1}^{n-1}jP_j-(n-1)\sum _{j=1}^{n}P_{n-1+j}-\sum _{j=1}^mQ_j-\sum _{j=0}^mR_j,\\
k_Z^*(P_{n-1-j})&=&P_{n-1+j},~j=0,\ldots ,n-2,~k_Z^*(P_{n-1+j})=P_{n-1-j},~j=0,\ldots ,n-2,\\
k_Z^*(P_{n-1+n-1})&=&H_Z-E_1-Q-\sum _{j=1}^{n-1}(j+1)P_j-n\sum _{j=1}^nP_{n-1+j}-E_2-\sum _{j=1}^mQ_j,\\
k_Z^*(P_{n-1+n})&=&E_2+C_2=H_Z-R_m,~k_Z^*(E_2)=P_{n-1+n},\\
k_Z^*(R_0)&=&C_4=2H_Z-E_1-2Q-\sum _{j=1}^{n-1}jP_j-(n-1)\sum _{j=1}^nP_{n-1+j}\\
&&-E_2-2\sum _{j=1}^mQ_j-R_m,\\
k_Z^*(R_j)&=&Q_j,~j=1,2,\ldots ,m,~k_Z^*(Q_j)=R_{j-1},~j=1,2,\ldots ,m.
\end{eqnarray*}
The spectral radius of $k_Z^*$ can be computed as the greatest real zero of the characteristic polynomial of the matrix representation $M_1$ of $k_Z^*$ restricted
to $\{H_Z,Q, P_{n-1+n-1},P_{n-1+n}, E_2$ ,$R_0,\ldots ,R_m$, $Q_1,\ldots
,Q_m\}$, which is
\begin{eqnarray*}
P(x)=-(x-1)x(x+1)[x^{2m}(x^3-nx^2-2x-1)+x^2+x+n].
\end{eqnarray*}
From this the conclusions of Lemma \ref{CharacteristicPolynomialNOdd4} follow.
\end{proof}

5. Case $n=deg (F)$ is odd; $a_0\not= 2/(m+1)$ for any $m=0,1,2,\ldots $; $a_0= \frac{n+1}{2}+\frac{l}{2(l+1)}$ for some $l=0,1,2,\ldots $;
$L_i(a_0,a_1,\ldots ,a_n)=0$ for any $0\leq i\leq n-2$.
\begin{lemma}
The spectral radius of $k_Z^*:Pic(Z)\rightarrow Pic(Z)$ is the largest real root of the polynomial $x^{2l+2}(x^3-nx^2-2x-1)+nx^2+x+1.$
\label{CharacteristicPolynomialNOdd5}\end{lemma}
\begin{proof}
Let $Z$ be the space constructed in Subcase 2.2 of Lemma \ref{HigherBlowup}. Denote
\begin{eqnarray*}
s_0&=&k_{Y_n}(ep_n)\in E_2,\\
t_1&=&k_{Y_n}^2(ep_n)\in P_{n-1+n},\\
\ldots&&\\
t_l&=&k_{Y_n}^{2l}(ep_n)\in P_{n-1+n},~s_l=k_{Y_n}^{2l+1}(ep_n)=0\in E_2.
\end{eqnarray*}
Let $P_{n-1+n+1}$ be the exceptional fiber of blowup at $ep_n$, $S_j$ the exceptional fiber of blowup at $s_j$, and $T_j$ the exceptional fiber of blowup
at $t_j$. Let $H_Z,E_1,Q,P_1,\ldots ,P_{n-1+n},P_{n-1+n+1},E_2$, $S_0,S_1,\ldots ,S_m,T_1,\ldots ,T_m$ be a basis for $Pic(Z)$. In this basis then
\begin{eqnarray*}
C_1&=&H_Z-E_1-Q-\sum _{j=1}^{n-1}(j+1)P_j-n\sum _{j=1}^{n+1}P_{n-1+j}-E_2-n\sum _{j=1}^lT_j-\sum _{j=0}^{l}S_j,\\
C_2&=&H_Z-E_1-Q-E_2-\sum _{j=0}^mS_j,\\
C_3&=&H_Z-E_1-Q-\sum _{j=1}^{n-1}jP_j-(n-1)\sum _{j=1}^{n+1}P_{n-1+j}-(n-1)\sum _{j=1}^lT_j,\\
C_4&=&2H_Z-E_1-2Q-\sum _{j=1}^{n-1}jP_j-(n-1)\sum _{j=1}^{n+1}P_{n-1+j}-E_2-(n-1)\sum _{j=1}^lT_j-\sum _{j=0}^{l-1}S_j-2S_m.
\end{eqnarray*}
To justify these formulas note that in the local coordinate for $E_2$ chosen in Lemma \ref{HigherBlowup}
\begin{eqnarray*}
C_1\cap E_2=[\infty ]_{E_2},~C_2\cap E_2=[1]_{E_2},~C_3\cap E_2=\emptyset ,~C_4\cap E_2=[0]_{E_2}.
\end{eqnarray*}
Then from the condition imposed on $a_0$, it follows that $s_0,\ldots ,s_m\notin C_1\cup C_2\cup C_3$, while $C_4$ goes through $s_m=[0]_{E_2}$ with
multiplicity $1$, and $s_0,\ldots ,s_{m-1}\notin C_4$. Then $k_Z^*:Pic(Z)\rightarrow Pic(Z)$
acts as
\begin{eqnarray*}
k_Z^*(H_Z)&=&(2n+1)H_Z-nE_1-(n+1)Q-(n+1)\sum _{j=1}^{n-1}jP_j\\
&&-\sum _{j=1}^{n-1}(n^2-1+j)P_{n-1+j}-(n^2+n-2)P_{n-1+n}-nE_2\\
&&-(n^2+n-2)\sum _{j=1}^lT_j-n\sum _{j=0}^{l-1}S_j-2nS_l ,\\
k_Z^*(E_1)&=&E_1,~k_Z^*(Q)=H_Z-E_1-Q-\sum _{j=1}^{n-1}jP_j-(n-1)\sum _{j=1}^{n}P_{n-1+j}-(n-1)\sum _{j=1}^lT_j,\\
k_Z^*(P_{n-1-j})&=&P_{n-1+j},~j=0,\ldots ,n-2,~k_Z^*(P_{n-1+j})=P_{n-1-j},~j=0,\ldots ,n-2,\\
k_Z^*(P_{n-1+n-1})&=&H_Z-E_1-Q-\sum _{j=1}^{n-1}(j+1)P_j-n\sum _{j=1}^nP_{n-1+j}-E_2-\sum _{j=0}^lS_j,\\
k_Z^*(P_{n-1+n})&=&E_2,~k_Z^*(E_2)=P_{n-1+n},\\
k_Z^*(P_{n-1+n+1})&=&H_Z-E_2-\sum _{j=0}^lS_j,\\
k_Z^*(S_0)&=&P_{n-1+n+1},~k_Z^*(S_j)=T_j,~j=1,\ldots ,l,~k_Z^*(T_j)=S_{j-1},~j=1,\ldots ,l.
\end{eqnarray*}
The spectral radius of $k_Z^*$ can be computed as the greatest real zero of the characteristic polynomial of the matrix representation $M_1$ of $k_Z^*$ restricted
to $\{H_Z,Q,P_{n-1+n-1},P_{n-1+n+1}$, $S_0,\ldots ,S_l$, $T_1,\ldots
,T_l\}$. The characteristic polynomial $P(x)=det(M_1-xI)$ of $M_1$ is
\begin{eqnarray*}
P(x)=-[x^{2l+2}(x^3-nx^2-2x-1)+nx^2+x+1].
\end{eqnarray*}
From this the conclusions of Lemma \ref{CharacteristicPolynomialNOdd5} follow.
\end{proof}

\end{document}